\documentclass{article}
\usepackage{makeidx}
\makeindex
\usepackage{amsmath}
\usepackage{amssymb}
\usepackage{amsthm}
\usepackage{fullpage}

\begin{document}
\title{Exact solutions and critical behaviour for a linear growth-diffusion equation on a time-dependent domain}
\date{}
\author{Jane Allwright}
\maketitle

\newtheorem{theorem}{Theorem}[section]
\newtheorem{lemma}{Lemma}[section]
\newtheorem{proposition}{Proposition}[section]
\newtheorem{corollary}{Corollary}[section]
\theoremstyle{definition}
\newtheorem{remark}{Remark}[section]
\newtheorem{example}{Example}[section]
\newtheorem{definition}{Definition}[section]

\def\ds{\displaystyle}
\def\Ai{\mathop{\rm Ai}\nolimits}

{\bf Abstract}
A linear growth-diffusion equation is studied in a time-dependent interval whose location and length both vary.
We prove conditions on the boundary motion for which the solution can be found in exact form, and derive the explicit expression in each case.

Next we prove the precise behaviour near the boundary in a `critical' case: when the endpoints of the interval move in such a way that near the boundary there is neither exponential growth nor decay, but the solution behaves like a power law with respect to time.
The proof uses a subsolution based on the Airy function with argument depending on both space and time.
Interesting links are observed between this result and Bramson's logarithmic term in the nonlinear FKPP equation on the real line.

Each of the main theorems is extended to higher dimensions, with a corresponding result on a ball with time-dependent radius.

\section{Introduction}
We consider the linear reaction-diffusion (or growth-diffusion) problem:
\begin{equation}\label{eq_psi}
\frac{\partial \psi}{\partial t} = D \frac{\partial^2 \psi}{\partial x^2} +f_0\psi  \qquad \textrm{in } A(t)< x< A(t)+L(t)
\end{equation}
\begin{equation}\label{eq_psi_BC}
\psi=0 \qquad \textrm \qquad\textrm{at } x=A(t) \textrm{ and } x=A(t)+L(t)
\end{equation}
where $\psi \geq 0$ (representing a population or concentration, for example). Here, $D>0$ is the diffusion coefficient, the constant $f_0>0$ describes the growth, and there are homogeneous Dirichlet boundary conditions at the endpoints. Both the start of the interval, $A(t)$, and the length of the interval, $L(t)$, are prescribed functions of time, and we assume that both $A(t)$ and $L(t)$ are twice continuously differentiable.

The scenario of a domain with moving boundaries is relevant in the context of, for example, a species population in a habitat which changes over time. This could be due to factors such as flooding, climate change, habitat destruction, forest fire, or `re-wilding' or `re-greening' areas of land. Such phenomena mean that the domain's size, as well as location, can change with time.
While that is one motivation, there are numerous other physical applications of diffusion equations,
and this work is relevant whenever these processes occur within a spatial domain whose boundary moves due to some external influence. (It is worth remarking that this is not the same as a free boundary problem, in which the moving boundary would be determined as part of the solution.)

We treat a linear growth term: $f_0\psi$ for some $f_0>0$. This allows for certain exact solutions and precise bounds on behaviour, which are useful for understanding and evaluating the effects of a time-dependent domain, as well as having mathematical interest. Following this, we intend to treat the case of a so-called FKPP-type nonlinear term (named from the initials of the authors of \cite{Fis} and \cite{KPP}), where $f(0)=f(1)=0$, $f>0$ on $(0,1)$, and $f(k)\leq f'(0)k$. This has applications to population dynamics. The results for the nonlinear case will be discussed elsewhere. An important property of such nonlinear terms is that the solution to the linear problem (with $f_0=f'(0)$) is a supersolution. Moreover, a linearisation around the zero state can be used as an approximation to the nonlinear equation when the population density is small enough. Therefore, a thorough understanding of this linear problem (with the full time-dependence) will also be an important tool in the analysis of nonlinear problems in time-dependent domains.

Due to the importance of climate change, and its consequences for the migration of species, the topic of habitat movement in a reaction-diffusion model has been considered by several authors. (See, for example, \cite{PotLew}, \cite{RoqRoqBerKre}, \cite{BerRos1}, \cite{BerRos2}, \cite{BerDieNagZeg}, \cite{BerDesDie}, \cite{BouNad}, \cite{AlfBerRao}, \cite{WuWanZou}, and \cite{BouGil}). 
To our knowledge, all of these make the mathematically convenient assumption that climate change translates the habitat at a constant speed $c$.
Here, we consider not only the case of a fixed length $L$ and constant speed $c$, but also several other much more general moving boundaries. The domain length is able to vary with time.
The results presented here will focus, primarily, on some particular cases of $A(t)$ and $L(t)$ for which exact results can be given: explicit expressions for the solution for certain forms of $L(t)$, and precise boundary behaviour in a `critical' case.
We also extend the methods to much more general forms of $A(t)$, $L(t)$, making use of a comparison principle on a transformed version of the equation. This provides useful upper or lower bounds on the solution for a range of cases.

This paper is split into two main sections. Section \ref{section_separable} treats the cases which can be solved exactly, deriving the explicit expressions and studying some of their implications.
In section \ref{section_crit} we construct a supersolution and an Airy function subsolution, to prove precise bounds on the solution near a boundary moving with $A(t)=-2\sqrt{Df_0}t+\alpha\log(t+1)+O(1)$. This  describes a `critical' case, in the sense that as $\alpha$ varies a transition occurs between growth and decay, with the solution near the boundary behaving like a power law in $t$.

We begin by transforming
onto a fixed spatial domain. We change variables from $\psi(x,t)$ to $u(\xi,t)$ where $\xi=\frac{(x-A(t))}{L(t)}L_0$, with $L_0=L(0)$, and
obtain the variable-coefficient PDE
\begin{equation}\label{eq_u}
\frac{\partial u}{\partial t} = D \frac{L_0^2 }{L(t)^2} \frac{\partial^2 u}{\partial \xi^2}+\frac{(\dot{A}( t)L_0+\xi\dot{L}(t))}{L(t)}\frac{\partial u}{\partial \xi} +f_0 u  \qquad \textrm{in } 0< \xi< L_0
\end{equation}
\begin{equation}
u=0 \qquad \textrm\qquad\textrm{at } \xi=0 \textrm{ and } \xi=L_0
\end{equation}

In section \ref{section_separable} we introduce a further change of variables, and thus deduce conditions under which the transformed equation can be solved exactly by separation of variables, reducing to a Sturm-Liouville problem on a fixed interval. These conditions are that $\ddot{L}L^3$ and $\ddot{A}L^3$ are constants, which is when the length varies as $L(t)=\sqrt{at^2 + 2bt+L_0^2}$. The forms of $A(t)$, and of the exact solutions, depend on whether $a$ is zero or non-zero, and on the sign of $aL_0^2-b^2$. We derive the explicit expressions for $u(\xi,t)$ in each exactly-solvable case and
describe several implications of the results.
We also extend the same method to a ball in $\mathbb{R}^n$ with moving centre and time-dependent radius $R(t)$, giving exact solutions when $\ddot{R}R^3$ is constant, i.e. $R(t)=\sqrt{at^2 + 2bt+R_0^2}$.

The exact solutions show directly each of the individual factors involved in determining the overall behaviour, and they determine precisely how any initial condition will evolve over time. This is very instructive in understanding the effect of each time-dependent domain. One also sees the effect of each parameter, which gives useful insight into how any changes or uncertainty in the parameters would alter the solution. We suggest that these exact solutions could be a useful tool in comparing theoretical with ecological observations. Finally, they are a means of deducing the long-time asymptotic behaviour: we show that in some cases the solution tends uniformly to zero, in other cases it becomes arbitrarily large at each interior point, while there are also cases for which the solution grows exponentially on part of the domain while decaying elsewhere. This third sort of behaviour occurs when the interface (between the regions for growth and decay) can travel at an asymptotic speed $c_*=2\sqrt{Df_0}$ while staying within the domain.
We recall that this speed $c_*$ is also the asymptotic spreading speed for solutions to the linear and nonlinear FKPP equations on the real line with compactly supported initial conditions (see \cite{KPP}, \cite{AroWei_1}), and it is the minimum wave speed for travelling wave solutions to the FKPP equation.

In section \ref{section_general_AL}, we use comparison principles to deduce upper and lower bounds on the solution for other forms of $L(t)$ and $A(t)$. One application gives bounds whenever the domain $(A(t), A(t)+L(t))$ encloses, or is enclosed by, one of the exactly-solvable cases. A second method allows us to derive bounds whenever $\ddot{L}L^3$ and $\ddot{A}L^3$ are bounded.

Our exact solutions are (it would seem) previously unknown.
A linear growth-diffusion on an expanding domain was analysed by Simpson in \cite{Sim} (and extended to a coupled system in \cite{SimShaMorBak}). In that model the domain was itself expanding at each position $x$, to model the uniform growth of living tissue. This differs to the case considered here, where the physical points inside the domain are not being expanded, but rather the boundary of the domain is moving. This led to a different reaction-diffusion equation in \cite{Sim}:
\begin{equation}
\frac{\partial C}{\partial t} = \frac{D}{L(t)^2} \frac{\partial^2 C}{\partial \xi^2}+ (k-\sigma(t))C  \qquad \textrm{in } 0< \xi< 1 \nonumber
\end{equation}
In our own model, the additional terms in equation \eqref{eq_u} mean that more changes of variables are required in order to get the equation into a separable form (see equation \eqref{eq_w}). Moreover, the dependence of these terms on both space and time means that it is only separable under the extra conditions that $\ddot{L}L^3$ and $\ddot{A}L^3$ are constants. Our explicit solutions appear more intricate than those in \cite{Sim}, although the principle is the same. Note, also, the paper \cite{JanRoy} which considered the Schr{\"{o}}dinger equation on an interval $(0,L(t))$ of changing length, and derived necessary conditions to solve it by separation of variables.

On the topic of exact solutions to certain parabolic equations, let us also mention the works \cite{SuaSusVeg_1}, \cite{SuaSusVeg_2} by Suazo, Suslov and Vega-Guzm\'{a}n. They used transformations of variables to convert between a diffusion-type equation with variable coefficients and the heat equation, and thus they derived the fundamental solution for their class of equation. This was given in terms of
the solution $\mu(t)$ to a second order ODE, and a set of six coefficients which were themselves defined by integrals involving $\mu$, $\mu'$, and the time-dependent coefficients of the parabolic equation.

Reaction-diffusion models on domains subject to translation at a constant speed $c$ have been considered by several authors.
We note in particular the paper by Potapov and Lewis \cite{PotLew} on a two-species competition, and the paper of Berestycki (H.), Diekmann, Nagelkerke and Zegeling \cite{BerDieNagZeg} for a single species (see also \cite{RoqRoqBerKre}, \cite{BerRos1}, \cite{BerRos2}, \cite{BerDesDie}, \cite{BouNad}, \cite{AlfBerRao}, \cite{WuWanZou}, and \cite{BouGil}). These two papers considered a nonlinear reaction term, and a model on the real line with growth in a favourable region --- of a fixed length $L$ and moving at a constant speed $c$ --- and decay elsewhere. The case of Dirichlet boundary conditions on a finite interval was included as a limiting case. Several interesting results were proved regarding the dynamics on a moving domain as opposed to a stationary one. (See, especially, the observations in \cite{PotLew} regarding invasibility in a moving domain.) Both papers proved the existence of a minimal domain length $L$ needed for survival, and expressed this as a function of $c$. If $c$ was greater than a certain critical value then the solution decayed exponentially to zero regardless of the domain length. The implication was that if the climate changes too rapidly then the species is unable to keep up, and goes extinct. This critical speed, $c_*=2\sqrt{Df_0}$, features in our solutions in a similar manner.

In section \ref{section_crit}, we consider the behaviour on a domain whose endpoints move close to the critical speed $c_*$.
An analysis of one of our exact solutions suggests that a logarithmic-in-time adjustment may be the key to this.
This is further motivated by the well-known result regarding Bramson's logarithmic correction 
in relation to the nonlinear FKPP equation on the real line with compactly supported initial conditions.
In that case, it has been proven that the positions $x=\pm (c_*t - \frac{3D}{c_*}\log(t+1)+O(1))$ are the asymptotic positions at which the solution takes on any value strictly between zero and the finite stable equilibrium. Moreover, there is locally uniform convergence, at this shifted position, to the profile of the minimum speed travelling wave. This result is known as Bramson's logarithmic correction. (See \cite{Bra_1}, \cite{Bra_2} for Bramson's original proof using probabilistic arguments, or \cite{HamNolRoqRyz_2} for an alternative proof using PDEs by Hamel, Nolen, Roquejoffre and Ryzhik.)

Here, we study the behaviour near the boundary when
\begin{equation}
A(t)=-L(t)/2=-c_*t+\alpha\log(t+1)+O(1)
\end{equation}
We construct super- and sub-solutions to demonstrate 
that when $\alpha>0$, the solution at $x=A(t)+y$ (for $y=O(1)$) is `exactly of order' $y t^{-\frac{3}{2}+\frac{\alpha c_*}{2D}}$ as $t\rightarrow\infty$. (A precise statement can be found in section \ref{section_crit}.)
In particular, the `critical' boundary motion, for which the solution at $A(t)+y$ remains exactly of order $y$, is $A(t)=-c_*t+\frac{3D}{c_*}\log(t+1)+O(1)$.
This precisely matches Bramson's logarithmic term.

Our analysis uses a change of variables; a supersolution based on the principal eigenfunction of the Laplacian; and a subsolution constructed from a space-and-time-dependent Airy function $\Ai$ and its tangent at the position $\Ai(0)$.

Bramson's logarithmic term (or similar) has been seen to arise in several other circumstances. We note in particular the paper \cite{Gar}, by G\"{a}rtner, which generalised the result to the multi-dimensional case (see also \cite{RoqRosRou}), and the paper \cite{JBer}, by Berestycki (J.), Brunet and Derrida, which derived the term in the setting of a linear equation on a semi-infinite interval with a free boundary. They prescribed constant values of the function and its gradient at the free boundary, and then calculated the precise asymptotics of the boundary motion for which the prescribed conditions would be satisfied.
Again, the leading term was $c_*t$ and the next term was of order $\log(t)$.  For initial conditions with suitable decay, the coefficient of the logarithmic term was the same as in Bramson's correction. (Many subsequent terms were also calculated; see \cite{JBer}.)

To our knowledge, this is the first time that such a term has appeared in the context of the linear equation on a finite, but time-dependent, interval with Dirichlet boundary conditions.
In contrast to our own method of super- and sub-solutions,
the logarithmic correction term in \cite{HamNolRoqRyz_2} was derived using bounds on some approximate solutions together with parabolic estimates in the function spaces $L^2$ and $H_0^1$, and in \cite{JBer} it was derived using a clever integral transform method and a singularity analysis in a small parameter. These three derivations of the term are completely different; nevertheless the same logarithmic term appears in each different setting.
It is possible that some useful insight into this somewhat `universal' logarithmic term may be gained from our change of variables --- which is the source of the factor $t^{-\frac{1}{2}+\frac{\alpha c_*}{2D}}$ in the critical behaviour --- or from our super- and sub-solutions --- the source of the $yt^{-1}$ factor.

We also discuss, in section \ref{section_crit_Rn}, the extensions of this result to a ball in $\mathbb{R}^n$ with radius $R(t)=c_*t-\alpha\log(t+1)+O(1)$.

\section{Exact solutions}\label{section_separable}
In this section we state and prove the form of each exact solution.
\begin{theorem}\label{theorem_exactsolutions}
Suppose that
\begin{equation}\label{eq_L}
L(t)^2=a t^2+2b t+L_0^2 \qquad \textrm{ for some } a,b,
\end{equation}
\begin{equation}\label{eq_ddotA}
\ddot{A}(t)=\frac{\gamma_{1}}{(a t^2+2b t+L_0^2)^{3/2}}\qquad \textrm{ for some } \gamma_{1}
\end{equation}
Then, for any given initial conditions $u(\xi,0)$ in $L^2([0,L_0])$, the solution for $u(\xi,t)$ can be obtained \emph{exactly}, as a sum of $u_n(\xi,t)$ with coefficients depending only on the initial conditions. The functions $u_n$ are given by
\begin{equation}\label{eq_un_theorem}
u_n(\xi,t)=\exp\left(\sigma_n \int\limits_0^ t \frac{L_0^2}{L(\zeta)^2} d\zeta\right)g_n(\xi)\left(\frac{L_0}{L(t)}\right)^{1/2}\exp{\left(f_0t -\int\limits_0^ t \frac{\dot{A}(\zeta)^2 }{4D} d\zeta-\frac{\xi^2 \dot{L}(t)L(t)}{4DL_0^2}-\frac{\xi\dot{A}(t)L(t)}{2DL_0}\right)}
\end{equation}
where
$g_n(\xi)$ satisfies the Sturm-Liouville problem in equations \eqref{eq_SL}, \eqref{eq_SL_BC} with  $\gamma_{0}=aL_0^2-b^2$, with eigenvalue $\sigma_n$.
The explicit expressions for these exact solutions depend on whether $a$ is zero or non-zero, and on the sign of $aL_0^2-b^2$. They are given in full in equations \eqref{eq_un_L0}, \eqref{eq_un_alphac}, \eqref{eq_un_rho}, \eqref{eq_un_g0negpos}, \eqref{eq_Th_g0neg}, and \eqref{eq_Th_g0pos}.
\end{theorem}
These explicit expressions determine precisely how any initial condition will evolve over time, and demonstrate each factor contributing to the behaviour. We can compare equation \eqref{eq_un_theorem} (or the specific formulae in equations \eqref{eq_un_L0}, \eqref{eq_un_alphac}, \eqref{eq_un_rho}, \eqref{eq_un_g0negpos}, \eqref{eq_Th_g0neg}, and \eqref{eq_Th_g0pos}) with the more standard case of a Fourier series solution on a fixed domain, for which
\begin{equation}\label{eq_fourier}
\tilde{u}_n(\xi,t)=\exp\left(-\frac{Dn^2\pi^2}{L_0^2}t\right) \sin\left(\frac{n\pi \xi}{L_0}\right)\exp(f_0t)
\end{equation}
The comparison is very instructive in understanding the precise effects of the time-dependent domain on the way the solution develops --- the subtleties of which would otherwise have been non-obvious.
\begin{proof}
We begin with a useful change of variables. With $u(\xi,t)$ satisfying equation \eqref{eq_u}, let
\begin{equation}\label{eq_w}
w(\xi, t)=u(\xi, t)\left(\frac{L(t)}{L_0}\right)^{1/2}\exp{\left(-f_0t +\int\limits_0^ t \frac{\dot{A}(\zeta)^2 }{4D} d\zeta+\frac{\xi^2 \dot{L}(t)L(t)}{4DL_0^2}+\frac{\xi\dot{A}(t)L(t)}{2DL_0}\right)}
\end{equation}
This removes the terms in $\frac{\partial u}{\partial \xi}$, shifting the effects of the time-dependent domain into the factors in \eqref{eq_w} and the zero-order term in the equation satisfied by $w(\xi, t)$:
\begin{equation}
\frac{\partial w}{\partial t} = D \frac{L_0^2 }{L(t)^2} \frac{\partial^2 w}{\partial \xi^2}+\left( \frac{\xi^2\ddot{L}(t)L(t)}{4DL_0^2} + \frac{\xi\ddot{A}(t)L(t)}{2DL_0} \right)w  \qquad \textrm{in } 0< \xi< L_0
\end{equation}
Next, change the time variable from $t$ to $s(t)= \int\limits_0^ t \frac{L_0^2}{L(\zeta)^2}d\zeta$, and write $v(\xi,s)=w(\xi, t)$. Then
\begin{equation}\label{eq_v}
\frac{\partial v}{\partial s} = D \frac{\partial^2 v}{\partial \xi^2}+\left( \frac{\xi^2\ddot{L}(t(s))L(t(s))^3}{4DL_0^4} + \frac{\xi\ddot{A}(t(s))L(t(s))^3}{2DL_0^3} \right)v  \qquad \textrm{in } 0< \xi< L_0
\end{equation}
\begin{equation} \label{eq_vBC}
v=0 \qquad \textrm\qquad\textrm{at } \xi=0 \textrm{ and } \xi=L_0
\end{equation}
Notice that the $v(\xi,s)$ equation is separable if and only if $\ddot{L}L^3=\gamma_{0}=$constant and $\ddot{A}L^3=\gamma_{1}=$constant.
This corresponds to $L(t)$ given by equation \eqref{eq_L}, with $\gamma_{0}=aL_0^2-b^2$, and $A(t)$ satisfying
equation \eqref{eq_ddotA} (which can be integrated twice to give $A(t)$). 
The $v(\xi,s)$ equation is then separable, with solutions of the form
$v(\xi,s)=\exp(\sigma s)g(\xi)$
where $g(\xi)$ satisfies the related Sturm-Liouville problem:
\begin{equation}\label{eq_SL}
\sigma g(\xi) = D g''(\xi)+\left(\frac{\gamma_{0}\xi^2}{4DL_0^4} + \frac{\gamma_{1}\xi}{2DL_0^3}\right)g(\xi)  \qquad \textrm{in } 0< \xi< L_0
\end{equation}
\begin{equation}\label{eq_SL_BC}
g=0 \qquad \textrm\qquad\textrm{at } \xi=0 \textrm{ and } \xi=L_0
\end{equation}
The Sturm-Liouville theory gives that there is a countably infinite set of eigenfunctions $g_{n}$ with eigenvalues $\sigma_{n}$, and that $v(\xi,s)$ has an eigenfunction expansion in terms of $v_n(\xi,s):=\exp(\sigma_n s)g_n(\xi)$, with coefficients depending only on the initial conditions. Thus, the solution for $u(\xi,t)$ is given exactly by a sum of
\begin{equation}
u_n(\xi,t)=\exp(\sigma_n s(t))g_n(\xi)\left(\frac{L_0}{L(t)}\right)^{1/2}\exp{\left(f_0t -\int\limits_0^ t \frac{\dot{A}(\zeta)^2 }{4D} d\zeta-\frac{\xi^2 \dot{L}(t)L(t)}{4DL_0^2}-\frac{\xi\dot{A}(t)L(t)}{2DL_0}\right)}
\end{equation}
where $s(t)= \int\limits_0^ t \frac{L_0^2}{L(\zeta)^2}d\zeta$. Thus equation \eqref{eq_un_theorem} is proved.

The required integrals (for $A(t)$, $s(t)$) depend on the specific form of $L(t)$: namely whether $a$ is zero or non-zero, and on the sign of $aL_0^2-b^2$. They can each be done by standard calculus, resulting in the expressions for $u_n$ in equations \eqref{eq_un_L0}, \eqref{eq_un_alphac}, \eqref{eq_un_rho}, \eqref{eq_un_g0negpos}, \eqref{eq_Th_g0neg}, and \eqref{eq_Th_g0pos}.
\end{proof}
In the following sections, the long-time behaviour of each solution is extracted based on the leading order terms. In certain cases the governing term depends on the eigenvalue $\sigma_n$.
Recall that, by Sturm-Liouville theory, the eigenvalues satisfy $\sigma_{n+1} \leq \sigma_{n}$, and the largest eigenvalue, $\sigma_1$, corresponds to an eigenfunction which is positive.
We know already that when $\gamma_0=\gamma_1=0$ then $\sigma_1=-\frac{D\pi^2}{L_0^2}$. In cases when $\gamma_0<0$, we need the following lemma when inferring the asymptotic behaviour.
\begin{lemma}
If $\gamma_{0}=-\rho^2<0$, then 
\begin{equation}\label{eq_sigmabound2}
\sigma_1 < -\frac{\vert\rho\vert}{2L_0^2} +\frac{\gamma_1^2}{4D\rho^2L_0^2}
\end{equation}
\end{lemma}
\begin{proof}
Write
$\ds{g_1(\xi)=e^{-\frac{1}{2}(\eta-\eta_0)^2}h(\eta)}$, where
$\ds{\eta=\sqrt{\frac{\vert\rho\vert}{2D}}\frac{\xi}{L_0}}$ and
$\ds{\eta_0=\frac{\gamma_{1}}{\vert\rho\vert^{3/2}\sqrt{2D}}}$.
This puts the equation into self-adjoint form
\begin{equation}\label{eq_h}
\frac{d}{d\eta}\left(h'(\eta)e^{-\left(\eta-\eta_0 \right)^2}\right) = \lambda h(\eta) e^{-\left(\eta-\eta_0\right)^2} \qquad \textrm{in } 0< \eta< \sqrt{\frac{\vert\rho\vert}{2D}}
\end{equation}
\begin{equation}
h=0 \qquad \textrm {at } \eta=0 \textrm{ and }  \eta= \sqrt{\frac{\vert\rho\vert}{2D}}
\end{equation}
where 
\begin{equation}
\lambda=1+\frac{2L_0^2}{\vert\rho\vert}\sigma_1 - \frac{\gamma_{1}^2}{2D\vert\rho\vert^3}
\end{equation}
Integrate equation \eqref{eq_h} over the interval, and recall that $h$ is positive, to deduce $\lambda<0$. This is equivalent to equation \eqref{eq_sigmabound2}.
\end{proof}
The following sections give the full expressions for $u_n(\xi,t)$, as well as their long-time behaviour. An exact formula is obtained for each form of $L(t)$ that is possible (depending on $a$ and $aL_0^2-b^2$). The integrals (for $A(t)$, $s(t)$) and the expressions occurring in equation \eqref{eq_un_theorem} differ between the cases. Thus, the formulae given below are the result of performing the necessary calculations and integrals, and substituting the relevant expressions into equation \eqref{eq_un_theorem}.

\subsection{$L(t)=L_0$}\label{section_L0}
For a fixed domain length $L_0$, we have $\ddot{L}L^3=\gamma_0=0$ and the separable cases are those where
\begin{equation}
A(t)=\frac{\gamma_{1}}{2L_0^3} t^2+c t+d \qquad \textrm{for some } c, d
\end{equation}
The separable solutions then have the form
\begin{equation}\label{eq_un_L0}
u_n(\xi,t)=\exp(\sigma_n t)g_n(\xi)\exp\left(f_0t-\frac{1}{4D}\left(\frac{\gamma_{1}^2}{3L_0^6} t^3+\frac{c\gamma_{1}}{L_0^3} t^2+c^2 t\right)-\frac{\xi}{2DL_0}\left(\frac{\gamma_{1}}{L_0^2} t+cL_0\right)\right)
\end{equation}
If $\gamma_{1}\neq0$ then, as $ t\rightarrow\infty$, $u(\xi,t)\rightarrow 0$ since the behaviour is dominated by the term
\begin{equation}
\exp\left(-\frac{\gamma_{1}^2}{12DL_0^6} t^3\right)
\end{equation}
If $\gamma_{1}=0$ then, as $ t\rightarrow\infty$, there is exponential growth or decay in the cases $f_0>\frac{D\pi^2}{L_0^2}+\frac{c^2}{4D}$ or $f_0<\frac{D\pi^2}{L_0^2}+\frac{c^2}{4D}$ respectively.
Indeed, the long time behaviour of $u_n$ is governed by
\begin{equation}
\exp\left(\sigma_n t+f_0 t-\frac{1}{4D}c^2 t  \right)
\end{equation}
where $\sigma_n =-\frac{Dn^2\pi^2}{L_0^2}$.

\subsection{$L(t)=L_0+\alpha t$ with $\alpha\neq0$}
When $L(t)=L_0+\alpha t$ then again $\ddot{L}L^3=\gamma_0=0$, but now the separable cases are those where
\begin{equation}
A(t)=\frac{\gamma_{1}}{2\alpha^2(L_0+\alpha t)}+c t+d  \qquad \textrm{for some } c, d
\end{equation}
The separable solutions then have the form
\begin{align} \label{eq_un_alphac}
u_n(\xi, t)=&\exp\left(\frac{\sigma_n L_0 t}{L_0+\alpha t}\right)g_n(\xi)\left(\frac{L_0}{L_0+\alpha t}\right)^{1/2}\exp(f_0t)\nonumber \\
&\times \exp\left(-\frac{1}{4D}\left(c^2 t-\frac{c\gamma_{1} t}{\alpha L_0(L_0+\alpha t)}-\frac{\gamma_{1}^2}{12\alpha^3}\left(\frac{1}{(L_0+\alpha t)^3} -\frac{1}{L_0^3}  \right)\right)\right)\nonumber\\
&\times \exp\left(-\frac{\xi^2\alpha(L_0+\alpha t)}{4DL_0^2}-\frac{\xi c(L_0+\alpha t)}{2DL_0}+\frac{\xi\gamma_{1}}{4DL_0\alpha(L_0+\alpha t)}\right)
\end{align}
If $\alpha>0$ then as $ t\rightarrow\infty$, the behaviour is asymptotically governed by
\begin{equation}
\exp\left( f_0t-\frac{1}{4D}\left(c+\frac{\xi\alpha}{L_0}\right)^2 t   \right)
\end{equation}
Recall that we defined
\begin{equation}
c_* = 2\sqrt{Df_0}
\end{equation}
Thus, if $-\alpha -c_* < c < c_*$, then there is a region of $\xi$ in which there is exponential growth: namely, where
\begin{equation}
\max\left(0,\frac{L_0}{\alpha}(-c_*-c)\right)<\xi<\min\left(L_0, \frac{L_0}{\alpha}(c_*-c)\right)
\end{equation}
Otherwise $u(\xi,t)$ decays to zero everywhere in $(0,L_0)$.

If instead $\alpha<0$, then $L(t)\rightarrow 0$ as $t\rightarrow -L_0/\alpha$, and in this limit, $u(\xi,t)\rightarrow 0$. Indeed, if $\gamma_{1}\neq0$, then the behaviour is governed by
\begin{equation}
\exp\left(\frac{\gamma_{1}^2}{48 D\alpha^3(L_0+\alpha t)^3}\right)
\end{equation}
which decays exponentially since $\alpha<0$. If $\gamma_{1}=0$ then the governing term is
\begin{equation}
\exp\left(\frac{\sigma_n L_0 t}{L_0+\alpha t}\right)
\end{equation}
where $\sigma_n =-\frac{Dn^2\pi^2}{L_0^2}<0$ and so again, $u(\xi,t)\rightarrow 0$.

\subsection{$L(t)=\sqrt{L_0^2+2\rho t}$ with $\rho\neq 0$}\label{section_rho}
If $L(t)^2=L_0^2+2\rho t$ then $\ddot{L}L^3=\gamma_{0}=-\rho^2<0$ and the separable cases are those where
\begin{equation}
A(t)=\frac{-\gamma_{1}\sqrt{L_0^2+2\rho t}}{\rho^2}+c t+d  \qquad \textrm{for some } c, d
\end{equation}
The separable solutions then have the form
\begin{align}\label{eq_un_rho}
u_n(\xi, t)=&\left(\frac{L_0^2+2\rho t}{L_0^2}\right)^{\ds{\frac{\sigma_n L_0^2}{2\rho}-\frac{1}{4}-\frac{\gamma_{1}^2}{8\rho^3 D}}}g_n(\xi)\exp\left(f_0 t-\frac{c^2}{4D} t+\frac{c\gamma_{1}}{2\rho^2D}\left(\sqrt{L_0^2+2\rho t}-L_0\right)\right) \nonumber \\
&\times  \exp\left(-\frac{\xi^2\rho}{4DL_0^2}+\frac{\xi\gamma_{1}}{2DL_0\rho}-\frac{\xi c\sqrt{L_0^2+2\rho t}}{2DL_0}\right)
\end{align}
Therefore if $\rho>0$ then, as $t\rightarrow\infty$,
there is exponential growth or decay
in the cases $f_0>\frac{c^2}{4D}$ or $f_0<\frac{c^2}{4D}$ respectively.
Indeed, the long time behaviour in this case is governed by
\begin{equation}
\exp\left(f_0t-\frac{c^2}{4D}t\right)
\end{equation}
If instead $\rho<0$ then $L(t)\rightarrow0$ as $t\rightarrow -L_0^2/2\rho$, and in this limit $u(\xi,t)\rightarrow 0$. This follows because the behaviour of each $u_n$ is governed by
\begin{equation}
\left(\frac{L_0^2+2\rho t}{L_0^2}\right)^{\ds{\frac{\sigma_n L_0^2}{2\rho}-\frac{1}{4}-\frac{\gamma_{1}^2}{8\rho^3 D}}}
\end{equation}
This implies that $u(\xi,t)\rightarrow 0$, by using the bound in equation \eqref{eq_sigmabound2}.

\subsection{$L(t)=\sqrt{a t^2 + 2b t + L_0^2}$ with $a\neq 0$ and $aL_0^2-b^2 \neq 0$}\label{section_g0negpos}
If $L(t)^2=a t^2 + 2b t + L_0^2$ then $\ddot{L}L^3=\gamma_{0}=aL_0^2-b^2$, and for $aL_0^2-b^2 \neq 0$ the separable cases are those where
\begin{equation}
A(t)=\frac{-\gamma_{1}}{b^2-aL_0^2}\sqrt{a t^2+2b t + L_0^2} + ct +d  \qquad \textrm{for some } c, d
\end{equation}
The separable solutions have the form
\begin{align}\label{eq_un_g0negpos}
u_n(\xi, t)=&\Theta_n(t)g_n(\xi)\left(\frac{L_0^2}{a t^2+2b t+L_0^2}\right)^{1/4} \nonumber \\
&\times   \exp\left(f_0 t-\frac{1}{4D}\left(\frac{\gamma_{1}^2a}{(b^2-aL_0^2)^2}+c^2\right) t + \frac{c\gamma_{1}}{2D(b^2-aL_0^2)}\left(\sqrt{a t^2+2b t+L_0^2}-L_0\right)\right) \nonumber \\
&\times   \exp\left(-\frac{\xi^2(a t+b) }{4DL_0^2}+ \frac{\xi\gamma_{1}(a t+b)}{2DL_0(b^2-aL_0^2)} -\frac{\xi c}{2DL_0}\sqrt{a t^2+2b t+L_0^2}\right)
\end{align}
where if $\gamma_0=aL_0^2-b^2 < 0$,
\begin{equation}\label{eq_Th_g0neg}
\Theta_n(t)=\left(\frac{\left(a t+b-\sqrt{b^2-aL_0^2}\right)\left(b+\sqrt{b^2-aL_0^2}\right)}{\left(b-\sqrt{b^2-aL_0^2}\right)\left(a t+b+\sqrt{b^2-aL_0^2}\right)}\right)^{\ds{\frac{\sigma_n L_0^2}{2\sqrt{b^2-aL_0^2}} - \frac{\gamma_{1}^2}{8D(b^2-aL_0^2)^{3/2}}}}
\end{equation}
and if $\gamma_0=aL_0^2-b^2 >0$ then
\begin{equation}\label{eq_Th_g0pos}
\Theta_n(t)=e^{
\left(\frac{\sigma_n L_0^2}{\sqrt{aL_0^2-b^2}} +\frac{\gamma_{1}^2}{4D(aL_0^2-b^2)^{3/2}} \right)
\left(\arctan\left(\frac{a t+b}{\sqrt{aL_0^2-b^2}}\right)-\arctan\left(\frac{b}{\sqrt{aL_0^2-b^2}}\right)\right)}
\end{equation}
If $L(t)$ remains positive for all $t>0$, then $a>0$ and the behaviour as $t\rightarrow\infty$ is governed by
\begin{equation}
\exp\left(f_0t-\frac{1}{4D}\left(c-\frac{\gamma_{1}\sqrt{a}}{(b^2-aL_0^2)}+\frac{\xi\sqrt{a}}{L_0}\right)^2 t\right)
\end{equation}
So, if $\ds{-c_*+\frac{\gamma_{1}\sqrt{a}}{(b^2-aL_0^2)}-\sqrt{a}<c<c_*+\frac{\gamma_{1}\sqrt{a}}{(b^2-aL_0^2)}}$, then there is a region of $\xi$ in which there is exponential growth: namely, where
\begin{equation}\label{eq_abL0}
\max\left(0,\frac{L_0}{\sqrt{a}}\left(-c_*-c+\frac{\gamma_{1}\sqrt{a}}{(b^2-aL_0^2)}\right)\right)<\xi<\min\left(L_0,\frac{L_0}{\sqrt{a}}\left(c_*-c+\frac{\gamma_{1}\sqrt{a}}{(b^2-aL_0^2)}\right)\right)
\end{equation}
Otherwise, there is exponential decay everywhere.

If instead $L(t)\rightarrow0$ in a finite time, then it must be that $aL_0^2-b^2 < 0$, and that $L(t)\rightarrow0$ as
\begin{equation}
t\rightarrow -\frac{1}{a}\sqrt{b^2-aL_0^2}-\frac{b}{a}
\end{equation}
In this limit, $u(\xi,t)\rightarrow0$ since the behaviour is governed by
\begin{equation}
\left(\frac{b+\sqrt{b^2-aL_0^2}}{a t+b+\sqrt{b^2-aL_0^2}}\right)^{\ds{\frac{\sigma_n L_0^2}{2\sqrt{b^2-aL_0^2}} - \frac{\gamma_{1}^2}{8D(b^2-aL_0^2)^{3/2}}+\frac{1}{4}}}
\end{equation}
This implies that $u(\xi,t)\rightarrow0$, by
using the bound in equation \eqref{eq_sigmabound2} with $\rho^2=b^2-aL_0^2=-\gamma_{0}$.

\subsection{Observed properties of the solutions}\label{section_properties}
These expressions are very instructive in understanding the
effects of a time-dependent domain on the solution. From them, one can observe the ways in which the exact nature of the time-dependence influences the solution, in both short and long time.
Although the formulae in equations \eqref{eq_un_L0}, \eqref{eq_un_alphac}, \eqref{eq_un_rho}, \eqref{eq_un_g0negpos}, \eqref{eq_Th_g0neg}, and \eqref{eq_Th_g0pos} differ, we note some common behaviour of these exact solutions in the asymptotic large time (or finite time) limit.

Firstly, whenever the domain length tends to zero in a finite time, the solution also tends to zero uniformly in $\xi$ (see sections \ref{section_rho} and \ref{section_g0negpos}). In each case this follows from an upper bound on the eigenvalue $\sigma_1$.

Note, also, that in the separable cases with $L(t)\rightarrow\infty$ as $t\rightarrow\infty$, the long-time behaviour does not depend on the eigenvalue $\sigma_1$. In these cases, $s(t)=o(t)$ and the term $\exp(\sigma_1 s(t))$ is not of leading order.

Next, note that the separable solutions share the property that there is exponential growth at any $\xi\in(0,L_0)$ such that
\begin{equation}\label{eq_c*growth}
-c_*< \lim_{t\rightarrow\infty}\frac{x(\xi,t)}{t} <c_*
\end{equation}
whereas there is exponential decay if
\begin{equation}\label{eq_c*decay}
\left\vert  \lim_{t\rightarrow\infty} \frac{x(\xi,t)}{t} \right\vert > c_*
\end{equation}
(Here $x(\xi,t)=\frac{\xi}{L_0}L(t)+A(t)$ is the original variable.)
This is, in some sense, similar to the behaviour of the solution, $\tilde{\psi}$ say, on the whole real line with initial conditions compactly supported in $[a,b]$:
\begin{equation}
\tilde{\psi}(x,t)= \frac{1}{\sqrt{4\pi Dt}}\int_a^b \tilde{\psi}(y,0) \exp\left(f_0t -\frac{(x-y)^2}{4Dt} \right)dy
\end{equation}
This spreads at the asymptotic speed $c_*$, in the sense that for $\vert c\vert <c_*$, $\tilde{\psi}(ct,t)\rightarrow\infty$ whereas $\sup_{\vert x\vert \geq c_*t}\tilde{\psi}(x,t)\rightarrow 0$ as $t\rightarrow\infty$.
It is well-known that $c_*$ is also the minimum wave speed for travelling wave solutions to the nonlinear FKPP equation, and that it is the asymptotic spreading speed for solutions to the same equation on the real line with compactly supported initial conditions (see \cite{KPP}, \cite{AroWei_1}).
Much work has been done on determining the exact behaviour associated with this spreading, and especially with respect to Bramson's logarithmic correction term (see \cite{Bra_1}, \cite{Bra_2}, \cite{HamNolRoqRyz_2}).
It is natural, therefore, to be interested in the exact behaviour of our solution $\psi(x,t)$ (in terms of the original variable $x$), at this critical interface between growth and decay. This is considered in section \ref{section_crit}.

Finally, note that in section \ref{section_L0} and the particular sub-case where $\gamma_1=0$, the problem has become that of an interval of fixed length $L_0$ moving at a constant speed $c$. Our result is in agreement with \cite{PotLew} and \cite{BerDieNagZeg} in deriving a critical domain length, which is defined by the equation $f_0=\frac{D\pi^2}{L_0^2}+\frac{c^2}{4D}$, and represents a threshold between decay and growth.

\subsection{Applications to more general $A(t)$, $L(t)$} \label{section_general_AL}
The preceding results are relevant not only to those specific forms of $A(t)$, $L(t)$ which led to the exact solutions.
The explicit expressions can also be used to deduce bounds on the solution for other, more general, forms of $A(t)$ and $L(t)$.

The parabolic comparison principle leads to the following result.
\begin{proposition}
Let $\psi_1(x,t)$ and $\psi_2(x,t)$ be the solutions with $A_1(t)$, $L_1(t)$, and $A_2(t)$, $L_2(t)$ respectively. If (for each $t$) $(A_1(t), A_1(t)+L_1(t))\subseteq (A_2(t), A_2(t)+L_2(t))$, then $\psi_1(x,t) \leq\psi_2(x,t)$ for $x \in (A_1(t), A_1(t)+L_1(t))$.
\end{proposition}
This therefore provides an explicit lower [or upper] bound for the solution $\psi$, whenever the domain contains [or is contained by] one of the separable cases.

A rather different extension of the method is to consider cases for which $\ddot{L}L^3$ and $\ddot{A}L^3$ are each bounded above, or bounded below. In this case we can bound the solution by expressions involving $A(t)$ and $L(t)$, together with the same Sturm-Liouville eigenfunctions and eigenvalues that occurred in the preceding sections.
\begin{proposition}
Suppose that
\begin{equation}\label{eq_bounds}
\ddot{L}(t)L(t)^3\leq \gamma_{0}^{+}\qquad \textrm{and}\qquad
\ddot{A}(t)L(t)^3 \leq \gamma_{1}^{+}
\end{equation}
for some constants $\gamma_{0}^{+}$, $\gamma_{1}^{+}$.
Then, for any given initial conditions $u(\xi,0)$ in $L^2([0,L_0])$, the solution $u(\xi,t)$ can be bounded above by a sum of the $u_n(\xi,t)$ in equation \eqref{eq_un_theorem}, where now the
$g_n(\xi)$, $\sigma_n$ satisfy the Sturm-Liouville problem in equations \eqref{eq_SL}, \eqref{eq_SL_BC} with $\gamma_{0}=\gamma_{0}^{+}$ and $\gamma_{1}=\gamma_{1}^{+}$.
If equation \eqref{eq_bounds} holds with both inequalities reversed, then the solution $u(\xi,t)$ can instead be bounded below by a sum of the $u_n(\xi,t)$.
\end{proposition}
\begin{proof}
The same changes of variables as in Theorem \ref{theorem_exactsolutions} leads to equations \eqref{eq_v}, \eqref{eq_vBC} for $v(\xi,s)$. Let $v^{+}(\xi,s)$ satisfy these same equations but with $\gamma_{0}^{+}$ in place of $\ddot{L}L^3$, and $\gamma_{1}^{+}$ in place of $\ddot{A}L^3$. (This is the separable problem which has just been considered.) Now, due to the special form equation \eqref{eq_v} and the positivity of the solutions, $v^{+}$ is a supersolution for $v$. The comparison principle can be applied to $v$, to
deduce that if $v(\xi,0)\leq v^{+}(\xi,0)$, then $v(\xi,s)\leq v^{+}(\xi,s)$ for all $s$. On changing variables back, we obtain the stated upper bound on $u(\xi,t)$.

If the inequalities in equation \eqref{eq_bounds} are reversed, and if $v^{+}(\xi,0)\leq v(\xi,0)$, then $v^{+}(\xi,s)$ is instead a subsolution, and thus we obtain the lower bound on $u(\xi,t)$.
\end{proof}

For the sake of completeness, we make the remark that, in the level of generality considered (i.e. $A(t)$ and $L(t)$ twice continuously differentiable), the domains --- and consequently the solutions --- will be hard to describe in any very general terms. Indeed, examples can be constructed with alternating growth and decay, such that the solution becomes both arbitrarily large and arbitrarily small over time. Such examples are not the focus of this paper.
 
\subsection{Exact solutions on a ball in higher dimension}
To conclude this section, we demonstrate that a similar process can lead to exact solutions to the problem on a ball in $\mathbb{R}^n$ with radius $R(t)$ and centre ${\bf A}(t)$. Consider the problem
\begin{equation}\label{eq_psiRn}
\frac{\partial \psi}{\partial t} = D \nabla^2 \psi+ f_0\psi  \qquad \textrm{in } \vert {\bf x}-{\bf A}(t)\vert <R(t)
\end{equation}
\begin{equation}\label{eq_psiRnBC}
\psi=0 \qquad \textrm{on } \vert {\bf x}-{\bf A}(t)\vert =R(t)
\end{equation}
A change of variables from ${\bf x}$ to ${\bf z}=\frac{({\bf x}-{\bf A}(t))}{R(t)}R_0$, with $R_0=R(0)$, and from $\psi({\bf x},t)$ to
\begin{equation}\label{eq_wRn}
w({\bf z}, t) = \psi({\bf x},t)\left(\frac{R(t)}{R_0}\right)^{\frac{n}{2}}\exp{\left(-f_0 t +\int\limits_0^t \frac{\vert{\bf \dot{A}}(\zeta)\vert^2 }{4D}d\zeta+\frac{ \dot{R}(t)R(t)}{4DR_0^2}\vert {\bf z}\vert^2+\frac{R(t)}{2DR_0}{\bf z \cdot \dot{A}}(t)\right)} 
\end{equation}
followed by $s(t)= \int\limits_0^t \frac{R_0^2}{R(\zeta)^2}d\zeta$, and $v({\bf z}, s)=w({\bf z}, t)$, leads to the equation
\begin{equation}
\frac{\partial v}{\partial s} = D \nabla^2 v +\left(\frac{\vert {\bf z}\vert ^2 \ddot{R}(t(s))R(t(s))^3}{4DR_0^4} + \frac{({\bf z}\cdot{\bf \ddot{A}}(t(s))) R(t(s))^3}{2DR_0^3} \right)  v \qquad \textrm{for }  \vert {\bf z} \vert < R_0
\end{equation}
\begin{equation}
v=0 \qquad\textrm{at } \vert {\bf z} \vert =R_0
\end{equation}
This is separable in $s$, $r=\vert {\bf z} \vert$, and $\theta$ (the angular co-ordinates) if $\ddot{R}R^3=\gamma_{0}=$constant and $\ddot{{\bf A}}R^3={\bf 0}$.
This corresponds to $R(t)^2=at^2+2bt+R_0^2$ for some constants $a$,$b$, and $\gamma_{0}=aR_0^2-b^2$;
and ${\bf A}(t)={\bf A_0}+{\bf c}t$ for some constant vectors ${\bf A_0}$ and ${\bf c}$. 
The solutions can then be expressed in terms of a sum of the eigenfunctions $v_l(r,\theta)=H_l(\theta)X_l(r)$ of 
\begin{equation}\label{eq_SL_n}
\sigma_l v_l(r,\theta) = D \nabla^2 v_l +\frac{r ^2 \gamma_{0}}{4DR_0^4}  v_l \qquad \textrm{on }  r < R_0
\end{equation}
\begin{equation}\label{eq_SL_BC_n}
v_l=0 \qquad\textrm{at } r =R_0
\end{equation}
which satisfy the correct periodicity in $\theta$ and which are non-singular at the origin $r=0$. This leads to the following theorem.
\begin{theorem}
Let $\psi({\bf x},t)$ satisfy equations \eqref{eq_psiRn}, \eqref{eq_psiRnBC} on the ball in $\mathbb{R}^n$ with radius $R(t)=\sqrt{at^2+2bt+R_0^2}$ and centre ${\bf A}(t)={\bf A_0}+{\bf c}t$. Then, for any suitable initial conditions $\psi({\bf x},0)$, the solution for $\psi({\bf x},t)$ can be obtained \emph{exactly}, as a sum of $\psi_l({\bf x},t)$ with coefficients depending only on the initial conditions. The $\psi_l$ are expressed purely in terms of ${\bf c}$, the constants occurring in $R(t)$, and the eigenfunctions and eigenvalues ($v_l$, $\sigma_l$) of the eigenvalue problem \eqref{eq_SL_n}, \eqref{eq_SL_BC_n} with $\gamma_{0}=aR_0^2-b^2$.
\end{theorem}
The explicit expressions are similar to the one-dimensional case, but note the dependence on $n$ in the factor $\left(\frac{R(t)}{R_0}\right)^{\frac{n}{2}}$ in equation \eqref{eq_wRn} as well as, of course, the dependence on $n$ in the eigenfunctions and eigenvalues. In $n$ dimensions we have the bound
$\sigma_1 <-\frac{n\rho}{2R_0^2}$ on the principal eigenvalue when $\gamma_0=-\rho^2<0$.
 
\section{Critical boundary motion}\label{section_crit}
\subsection{Behaviour near the critical speed}
In this section we take up the question mentioned in section \ref{section_properties}, regarding
the exact behaviour of our solution $\psi(x,t)$ at the critical interface between growth and decay.
Recall that we defined
\begin{equation}
c_*=2\sqrt{Df_0}
\end{equation}
and that, in the separable solutions, there was exponential growth of $u(\xi,t)$ at any $\xi\in(0,L_0)$ such that equation \eqref{eq_c*growth} held,
and exponential decay at any $\xi$ where equation \eqref{eq_c*decay} held (i.e. $x(\xi,t)$ travelling slower or faster than $c_*$, respectively).
It is natural to seek a more precise description of this changeover between regions of growth and decay.
Given $A(t)$, $L(t)$, are we able to track the position $x(t)$ at which the solution $\psi(x,t)$ is equal to some a constant, $O(1)$, value? For which choices of $A(t)$, $L(t)$ will the solution be exactly of order $1$ near the boundary (neither growing to $\infty$ nor decaying to $0$)?

Initially, let us make use of an exact solution from section \ref{section_separable}. Let $\hat{\psi}(x,t)$ be the solution on the interval 
\begin{equation}
-c_*t - \frac{L_0}{2} < x <\frac{L_0}{2}+c_*t
\end{equation} 
This is given by equation \eqref{eq_un_alphac} with $c=-c_*$, $\alpha=2c_*$, $\gamma_1=0$, $\sigma_n=-\frac{Dn^2\pi^2}{L_0^2}$ and $g_n(\xi)=\sin\left(\frac{n\pi\xi}{L_0}\right)$. For simplicity, take the initial conditions to be $\sin\left(\frac{\pi\xi}{L_0}\right)$. Then (recalling that $f_0=\frac{c_*^2}{4D}$) this exact solution is:
\begin{align}\label{eq_uexact}
\hat{u}(\xi, t)=&\exp\left(-\frac{D\pi^2 t}{L_0( L_0+2c_* t) }\right)\sin\left(\frac{\pi\xi}{L_0}\right)\left(\frac{L_0}{L_0+2c_* t}\right)^{1/2}\exp\left(\frac{\xi c_*}{2DL_0}(L_0+2c_* t) \left(1-\frac{\xi}{L_0}\right)\right)
\end{align}
Therefore at $x=-c_*t - L_0/2 +y$, we have
\begin{align}
\hat{\psi}(-c_*t - L_0/2+y,t)=&\hat{u}\left(\frac{yL_0}{L_0+2c_*t},t\right)\nonumber\\
=&\exp\left(-\frac{D\pi^2 t}{L_0( L_0+2c_* t) }\right) \sin\left(\frac{\pi y}{L_0+2c_*t} \right) \left(\frac{L_0}{L_0+2c_*t}\right)^{1/2}\nonumber\\
&\times\exp\left(\frac{yc_*}{2D}\left(1-\frac{y}{L_0+2c_*t}\right)\right)
\end{align}
If $y=O(1)$ then as $t\rightarrow\infty$,
\begin{equation}
\hat{\psi}(-c_*t - L_0/2+y,t)=O\left( \frac{y}{t}\times \frac{1}{t^{1/2}}\times \exp\left(\frac{yc_*}{2D}\right)\right) = O(t^{-3/2})\rightarrow 0
\end{equation}
Observe that the choice $y(t)=\frac{3D}{c_*}\log(t+1)$ in equation \eqref{eq_uexact} removes all the powers of $t$, and gives (as $t\rightarrow\infty$): 
\begin{equation}
\hat{\psi}\left(-c_*t - \frac{L_0}{2}+\frac{3D}{c_*}\log(t+1),t\right)= O\left( \frac{\log t}{t}\times \frac{1}{t^{1/2}}\times (t+1)^{3/2} \right)=O( \log t) \rightarrow \infty
\end{equation}
The form of this exact solution for $\hat{\psi}$ suggests that the critical choices of $A(t)$, $L(t)$ (where the solution near to the boundary remains exactly of order $1$) may occur when the endpoints move as $\pm c_* t$ plus a logarithmic term (plus smaller order corrections). Furthermore, the fact that the choice $y(t)=\frac{3D}{c_*}\log(t+1)$ removes all the powers of $t$, suggests the likely coefficient of such a term.
The following section will give the precise statement of the behaviour on an interval which does include a logarithmic adjustment to the endpoints.

\subsection{Precise behaviour in the critical case}
From now on, we restrict attention to cases where $A(t)=-\frac{L(t)}{2}$. Our change of variables from equation \eqref{eq_w} becomes
\begin{equation}\label{eq_wu}
w(\xi, t)=u(\xi, t)\left(\frac{L(t)}{L_0}\right)^{1/2}\exp{\left(-f_0 t + \int\limits_0^ t \frac{\dot{L}(\zeta)^2 }{16D} d\zeta+\xi (\xi-L_0) \frac{\dot{L}(t)L(t)}{4DL_0^2}\right)}
\end{equation}
Let us give a precise definition of the behaviour we are interested in, and the notation we shall use for it.
\begin{definition}
Given two functions $F_1$, $F_2$, one will be referred to as being \emph{exactly of the order} the other (in a given limit), and denoted by $F_1=\overline{\underline{O}}(F_2)$, when $F_2=O(F_1)$ and $F_1=O(F_2)$ (in the limit under consideration). In other words, there are positive constants $0< \beta_0 \leq \beta_1$ such that $\beta_0  \vert F_2 \vert \leq \vert F_1 \vert \leq \beta_1 \vert F_2 \vert$.
\end{definition}
The following theorem is the main result which is proved in the remainder of this section. The proof relies on the construction of a supersolution and a subsolution, both having the specified behaviour.
\begin{theorem}\label{theorem_crit}
Let
\begin{equation}\label{eq_Acrit}
A(t)=\frac{-L(t)}{2}=-c_*t + \alpha\log(t+1) +\eta(t)
\end{equation}
where $c_*=2\sqrt{Df_0}$, $\alpha>0$, and
\begin{equation}\label{eq_eta}
\eta(t)=O(1), \qquad \dot{\eta}(t)=o(1/t), \qquad \ddot{\eta}(t)=o(1/t^2) \qquad \dddot{\eta}(t)=o(1/t^3)\qquad \textrm{as } t\rightarrow\infty
\end{equation}
Then
\begin{equation}
\psi(A(t)+y,t)=\overline{\underline{O}}\left(y t^{-\frac{3}{2}+\frac{\alpha c_*}{2D}} \right) \qquad \textrm{as }t\rightarrow\infty \textrm{, for }y=O(1)
\end{equation}
\end{theorem}
\begin{remark}
Note that the conditions on $\eta(t)$ allow, for example, $\eta(t)=\eta_0=$constant, or $\eta(t)=(t+1)^k$ for $k<0$, but not things like $\eta(t)=\log\log(t)$ as $t\rightarrow\infty$.
\end{remark}

This theorem gives asymptotic bounds on $\psi(x,t)$ for $x$ within $O(1)$ of the moving boundary. Hence, it also bounds the asymptotic behaviour of the gradient at the moving boundary itself:
\begin{equation}
\frac{\partial\psi}{\partial x}\left(\frac{-L(t)}{2},t\right)=\overline{\underline{O}}((t+1)^{-\frac{3}{2}+\frac{\alpha c_{*}}{2D}})\qquad \textrm{ as }t\rightarrow\infty
\end{equation}
In particular, at the critical value
\begin{equation}
\alpha=\alpha_{crit} = \frac{3D}{c_{*}}
\end{equation}
we have that $\psi\left(\frac{-L(t)}{2}+y,t\right)= \overline{\underline{O}}(y)$ as $t\rightarrow\infty$, and that the gradient at the boundary is bounded above and below independently of time: $\frac{\partial\psi}{\partial x}\left(\frac{-L(t)}{2},t\right)=\overline{\underline{O}}(1)$.

The derivation of the $\frac{3D}{c_{*}}\log(t+1)$ term in this context is completely different from the proofs in the other settings in which such a term arises. In this case (of a linear equation on a finite interval with moving boundaries) our derivation of the term is relatively straightforward, or accessible, being based solely on explicit super- and sub-solutions to a linear equation. Moreover, the bulk of our proof is in fact taken up in showing that the function $w(\xi,t)$ is exactly of order $\xi$ (or $y/t$). The other factor, $t^{-\frac{1}{2}+\frac{\alpha c_*}{2D}}$, in the critical behaviour comes straight from the change of variables.
This observation, and the exact expression used in the change of variables (equation \eqref{eq_wu}), may therefore help to give insight into the source of the logarithmic term in other settings.

Recall that the function $w$ now satisfies
\begin{equation}\label{eq_wcrit}
\frac{\partial w}{\partial t} = D \frac{L_0^2 }{L(t)^2} \left( \frac{\partial^2 w}{\partial \xi^2}+P(t) \frac{\xi}{L_0}\left( \frac{\xi}{L_0}-1\right)\frac{w}{L_0^2}\right)  \qquad \textrm{in } 0< \xi< L_0
\end{equation}
\begin{equation}\label{eq_wcritBC}
w=0 \qquad \textrm\qquad\textrm{at } \xi=0 \textrm{ and } \xi=L_0
\end{equation}
where
\begin{equation}\label{eq_Pt}
P(t)= \frac{\ddot{L}(t)L(t)^3}{4D^2}
\end{equation}
The following two propositions give a supersolution and a subsolution for $w(\xi,t)$ under certain conditions on $P(t)$.
It is worth noting that Proposition \ref{proposition_supersol} and Proposition \ref{proposition_subsol} apply in general whenever $w(\xi,t)$ satisfies equations \eqref{eq_wcrit}, \eqref{eq_wcritBC} for any function $P(t)$ (satisfying the conditions of the proposition). They do not rely at all on the specific form of $P(t)$ that we are interested in here, given by equation \eqref{eq_Pt}.
In the case where $P(t)$ \emph{is} given by equation \eqref{eq_Pt}, the condition \eqref{eq_conditions2} in Proposition \ref{proposition_subsol} becomes simply
\begin{equation}
\int_0^{\infty} \ddot{L}(\zeta)^{2/3} d\zeta < \infty 
\end{equation}

\begin{proposition} (Supersolution)\label{proposition_supersol}

Let $w(\xi,t)$ satisfy equations \eqref{eq_wcrit}, \eqref{eq_wcritBC}. If $P(t)\geq 0$ then (up to multiplication by a constant) $w(\xi,t)\leq \overline{w}(\xi,t)$ where \begin{equation}\label{eq_overlinew}
\overline{w}(\xi,t) =\sin \left(\frac{\pi\xi}{L_0}\right) \exp\left( -\int_0^t \frac{D\pi^2}{L(\zeta)^2}d\zeta \right)
\end{equation}
Moreover, if
\begin{equation}\label{eq_conditions1}
\int_0^{\infty} \frac{1}{L(\zeta)^2} d\zeta < \infty
\end{equation}
then $w(\xi,t)=O(\xi)$ independently of time as $t\rightarrow \infty$,
in the sense that given $B_1\in (0,L_0)$, there exists $\beta_1$ such that
\begin{equation}\label{eq_wbeta1}
w(\xi,t)\leq \beta_1 \xi   \textrm{ as } t\rightarrow\infty, \textrm{ for all }0\leq\xi\leq B_1
\end{equation}
\end{proposition}
\begin{proof}
The function $\overline{w}(\xi,t)$ satisfies the boundary conditions and, since $P(t)\geq0$, it satisfies the inequality
\begin{align}
\frac{\partial \overline{w}}{\partial t} &= D \frac{L_0^2 }{L(t)^2}\frac{\partial^2 \overline{w}}{\partial \xi^2}\nonumber \\
& \geq D \frac{L_0^2 }{L(t)^2} \left( \frac{\partial^2 \overline{w}}{\partial \xi^2}+P(t) \frac{\xi}{L_0}\left( \frac{\xi}{L_0}-1\right)\frac{\overline{w}}{L_0^2}\right) 
\end{align}
and so it is a supersolution for $w(\xi,t)$. Hence, up to multiplication by a constant, $w(\xi,t)\leq \overline{w}(\xi,t)$.
Moreover, if equation \eqref{eq_conditions1} holds,
then $\overline{w}(\xi,t)=O(\xi)$ (independently of time as $t\rightarrow\infty$) and so equation \eqref{eq_wbeta1} is proved.
\end{proof}

Next we construct a subsolution using the Airy function $\Ai$ 
and its tangent at the position $\Ai(0)$.
\begin{proposition} (Subsolution) \label{proposition_subsol}

Let $w(\xi,t)$ satisfy equations \eqref{eq_wcrit}, \eqref{eq_wcritBC}. If $P(t)\rightarrow\infty$ as $t\rightarrow\infty$ and $\dot{P}(t)\geq 0$, 
then (up to multiplication by a constant) $w(\xi,t)\geq\tilde{w}(\xi,t)= \underline{w}(\xi,t)a(t)$ where $\underline{w}(\xi,t)$ and $a(t)$ are given by equations \eqref{eq_underlinew} and \eqref{eq_at}.
Moreover, if
\begin{equation}\label{eq_conditions2}
\int_0^{\infty} \frac{P(\zeta)^{2/3}}{L(\zeta)^2} d\zeta < \infty
\end{equation}
then for $P(t)^{1/3} \frac{\xi}{L_0}$ sufficiently small, $w(\xi,t)$ can be bounded below by a positive multiple of $\xi$ (independently of $t$) as $t\rightarrow\infty$. In other words, for $B_0>0$ small enough, there exists $\beta_0>0$ such that
\begin{equation} \label{eq_wbeta0}
\beta_0 \xi \leq w(\xi,t) \textrm{ as } t\rightarrow\infty, \textrm{ for all }0\leq\xi\leq B_0P(t)^{-1/3}L_0 
\end{equation}
\end{proposition}
\begin{proof}
Let $c_1$ be the largest real zero of the Airy function $\Ai$. Note, for reference, the facts that $c_1<0$, $\Ai'(c_1)>0$, $\Ai(0)>0$, $\Ai'(0)<0$, and $\Ai''(0)=0$.
Define $\underline{w}(\xi,t)$ by:
\begin{equation}\label{eq_underlinew}
\underline{w}(\xi,t)=\begin{cases}\ds{\frac{1}{P(t)^{1/3}}}\Ai\left( P(t)^{1/3} \frac{\xi}{L_0} +c_1\right)\qquad\textrm{for }0\leq\xi\leq -c_1 P(t)^{-1/3}L_0 \textrm{: Region I}\\
\ds{\frac{1}{P(t)^{1/3}}}\left( \Ai(0)+\Ai'(0)\left( P(t)^{1/3} \frac{\xi}{L_0} +c_1\right)\right) \\
\qquad\textrm{for }  -c_1 P(t)^{-1/3}L_0  \leq \xi \leq -\left( \frac{\Ai(0)}{\Ai'(0)} +c_1  \right) P(t)^{-1/3}L_0 \textrm{: Region II}\\
0 \qquad \textrm{for } -\left( \frac{\Ai(0)}{\Ai'(0)} +c_1  \right) P(t)^{-1/3}L_0\leq \xi \leq L_0\textrm{: Region III}\\
\end{cases}
\end{equation}

Note that $\underline{w}$ is continuous and non-negative on $[0,L_0]$, and satisfies the boundary conditions. Furthermore, both $\frac{\partial \underline{w}}{\partial \xi}$ and $\frac{\partial^2 \underline{w}}{\partial \xi^2}$ are continuous across Regions I-II, including at the point where they meet, since at this point the left and right limits both give $\frac{\partial \underline{w}}{\partial \xi}=\frac{\Ai'(0)}{L_0}$ and $\frac{\partial^2 \underline{w}}{\partial \xi^2}=0$. In each Region I and Region II,
$\frac{\partial \underline{w}}{\partial t}$ satisfies
\begin{equation}
\frac{\partial \underline{w}}{\partial t} =\frac{\dot{P}(t)}{3P(t)}\left(-\underline{w} + \xi\frac{\partial \underline{w}}{\partial \xi}\right)
\end{equation}
and so it follows from the continuity of each term that $\frac{\partial \underline{w}}{\partial t}$ is also continuous across Regions I-II.

In Region I:
\begin{align}
\frac{\partial \underline{w}}{\partial t} -& D \frac{L_0^2 }{L(t)^2} \left( \frac{\partial^2 \underline{w}}{\partial \xi^2}+P(t) \frac{\xi}{L_0}\left( \frac{\xi}{L_0}-1\right)\frac{\underline{w}}{L_0^2}\right) \nonumber \\
= & -\frac{\dot{P}(t)}{3P(t)}\underline{w} + 
\frac{\dot{P}(t)}{3P(t)} \frac{\xi}{L_0} \Ai'\left( P(t)^{1/3} \frac{\xi}{L_0} +c_1  \right) \nonumber \\
& - \frac{DP(t)^{1/3}}{L(t)^2}\Ai''\left(  P(t)^{1/3} \frac{\xi}{L_0} +c_1 \right) - \frac{DP(t)}{L(t)^2}\frac{\xi^2}{L_0^2}\underline{w}+\frac{DP(t)}{L(t)^2}\frac{\xi}{L_0}\underline{w} \\
= & \frac{\dot{P}(t)}{3P(t)}\left(-\underline{w} + \xi\frac{\partial \underline{w}}{\partial \xi}\right)- \frac{D}{L(t)^2}P(t)^{2/3}\left(  P(t)^{1/3} \frac{\xi}{L_0} +c_1 \right)\underline{w}  \nonumber \\
& - \frac{DP(t)}{L(t)^2}\frac{\xi^2}{L_0^2}\underline{w}+\frac{DP(t)}{L(t)^2}\frac{\xi}{L_0}\underline{w} \\
=& \frac{\dot{P}(t)}{3P(t)}\left(-\underline{w} + \xi\frac{\partial \underline{w}}{\partial \xi}\right) -\frac{DP(t)^{2/3}}{L(t)^2}c_1\underline{w}  - \frac{DP(t)}{L(t)^2}\frac{\xi^2}{L_0^2}\underline{w} \label{eq_wxi}
\end{align}
Note that $\frac{\partial^2 \underline{w}}{\partial \xi^2}\leq 0$ in Region I, since $\Ai''(x)=x\Ai(x)\leq 0$ on $[c_1,0]$. Therefore, using $\underline{w}(0,t)=0$, it holds that
\begin{equation}
\xi\frac{\partial \underline{w}}{\partial \xi}(\xi,t) \leq \underline{w}(\xi,t) \qquad \textrm{in Region I}
\end{equation}
Thus equation \eqref{eq_wxi} together with the assumption that $P(t)\geq 0$ and $\dot{P}(t)\geq 0$ implies that, in Region I,
\begin{equation}
\frac{\partial \underline{w}}{\partial t} - D \frac{L_0^2 }{L(t)^2} \left( \frac{\partial^2 \underline{w}}{\partial \xi^2}+P(t) \frac{\xi}{L_0}\left( \frac{\xi}{L_0}-1\right)\frac{\underline{w}}{L_0^2}\right)  \leq 
-c_1\frac{DP(t)^{2/3}}{L(t)^2}\underline{w}
\end{equation}
In Region II, since $P(t)\geq 0$, $\dot{P}(t)\geq 0$, and $\Ai'(0)<0$,
\begin{align}
\frac{\partial \underline{w}}{\partial t} -& D \frac{L_0^2 }{L(t)^2} \left( \frac{\partial^2 \underline{w}}{\partial \xi^2}+P(t) \frac{\xi}{L_0}\left( \frac{\xi}{L_0}-1\right)\frac{\underline{w}}{L_0^2}\right)\nonumber\\
 = & -\frac{\dot{P}(t)}{3 P(t)}\underline{w}+ \frac{\dot{P}(t)}{3 P(t)}\frac{\xi}{L_0}\Ai'(0)-  \frac{DP(t)}{L(t)^2}\frac{\xi^2}{L_0^2}\underline{w} + \frac{DP(t)}{L(t)^2}\frac{\xi}{L_0}\underline{w} \\
\leq& \frac{DP(t)}{L(t)^2}\frac{\xi}{L_0}\underline{w} \\
\leq &\left( -\frac{\Ai(0)}{\Ai'(0)} -c_1  \right) \frac{DP(t)^{2/3}}{L(t)^2}   \underline{w}
\end{align}
This leads us to define $\tilde{w}(\xi,t)= \underline{w}(\xi,t)a(t)$ where
\begin{equation}\label{eq_at}
a(t)=\exp {\left(  \left( \frac{\Ai(0)}{\Ai'(0)} +c_1  \right)\int_0^t   \frac{DP(\zeta)^{2/3}}{L(\zeta)^2} d\zeta \right)}
\end{equation}
Then in Regions I-II, the function $\tilde{w}(\xi,t)$ is $C^2$ in $\xi$, $C^1$ in $t$ and it satisfies
\begin{equation}
\frac{\partial \tilde{w}}{\partial t} - D \frac{L_0^2 }{L(t)^2} \left( \frac{\partial^2 \tilde{w}}{\partial \xi^2}+P(t) \frac{\xi}{L_0}\left( \frac{\xi}{L_0}-1\right)\frac{\tilde{w}}{L_0^2}\right) \leq 0
\end{equation}
so it is a classical subsolution for $0\leq\xi\leq -\left( \frac{\Ai(0)}{\Ai'(0)} +c_1  \right) P(t)^{-1/3}L_0$ (Regions I-II).

It is clear that $\tilde{w}\equiv 0$ is also a classical subsolution in Region III.
At the point where Region II and Region III meet, $\tilde{w}$ is continuous, it is a classical subsolution on either side, and $\frac{\partial \tilde{w}}{\partial \xi}$ has a jump discontinuity from a negative value on the left (Region II) to zero on the right (Region III). It follows that $\tilde{w}(\xi,t)$ is a weak subsolution to the parabolic problem on $(0,L_0)$.
Therefore, up to multiplication by some constant,
\begin{equation}
\tilde{w}(\xi,t)\leq w(\xi,t)
\end{equation}

If equation \eqref{eq_conditions2} holds,
then $a(t)$ converges to a strictly positive value as $t\rightarrow\infty$. Then, since
\begin{equation}
\underline{w}(\xi,t) \sim \frac{\Ai'(c_1)}{L_0}\xi \qquad\textrm{as }P(t)^{1/3} \frac{\xi}{L_0}\rightarrow 0
\end{equation}
it follows that for $P(t)^{1/3} \frac{\xi}{L_0}$ sufficiently small, $\tilde{w}(\xi,t)$ can be bounded above and below by positive multiples of $\xi$ (independently of time as $t\rightarrow\infty$). In particular, for $B_0>0$ small enough, there exists $\beta_0>0$ such that
\begin{equation}
\beta_0 \xi \leq \tilde{w}(\xi,t) \textrm{ as } t\rightarrow\infty, \textrm{ for all }0\leq\xi\leq B_0P(t)^{-1/3}L_0 
\end{equation}
Equation \eqref{eq_wbeta0} follows and the proposition is proved.
\end{proof}
Next we use the super- and sub-solutions for $w(\xi,t)$ to prove Theorem \ref{theorem_crit}:
\begin{proof} (of Theorem \ref{theorem_crit})

Recall from equation \eqref{eq_Acrit} that
\begin{equation}\label{eq_Lcrit}
L(t)=2(c_*t - \alpha\log(t+1) -\eta(t))
\end{equation}
Thus, as $t\rightarrow\infty$ the function $\ds{P(t)=\frac{\ddot{L}(t)L(t)^3}{4D^2}}$ obeys
\begin{equation}
P(t)\sim \frac{4 \alpha c_{*}^3 }{D^2}t \rightarrow\infty \qquad \textrm{and} \qquad  \dot{P}(t) \sim \frac{4\alpha c_{*}^3 }{D^2} >0
\end{equation}
Moreover, since $L(t)\sim 2c_* t$ and $\ddot{L}(t)\sim 2\alpha t^{-2}$ as $t\rightarrow\infty$, it also holds that
\begin{equation}
\int_0^{\infty} \frac{1}{L(\zeta)^2} d\zeta < \infty \qquad \textrm{and} \qquad \int_0^{\infty} \ddot{L}(\zeta)^{2/3} d\zeta < \infty
\end{equation}
So, both Proposition \ref{proposition_supersol} and \ref{proposition_subsol} apply to this case, giving that for some positive constants $C_1$ and $C_2$, $C_1\underline{w}(\xi,t)a(t)\leq w(\xi,t)\leq C_2\overline{w}(\xi,t)$, and that
for $B_0>0$ small enough, there exist $0<\beta_0\leq\beta_1$ such that
\begin{equation}
\beta_0 \xi \leq \tilde{w}(\xi,t)\leq \beta_1\xi \textrm{ as } t\rightarrow\infty, \textrm{ for all }0\leq\xi\leq B_0P(t)^{-1/3}L_0 
\end{equation}
Hence, we have shown that $w(\xi,t)$ is exactly of order $\xi$:
\begin{equation}
w(\xi,t)=\overline{\underline{O}}(\xi)\qquad\textrm{ as }\xi=O(P(t)^{-1/3} )\rightarrow 0
\end{equation}
In terms of the original function $\psi(x,t)$, recall that
\begin{equation}
\psi(x,t)=u(\xi, t)=w(\xi, t)\left(\frac{L_0}{L(t)}\right)^{1/2}\exp{\left(f_0 t - \int\limits_0^ t \frac{\dot{L}(\zeta)^2 }{16D} d\zeta - \xi (\xi-L_0) \frac{\dot{L}(t)L(t)}{4DL_0^2}\right)} 
\end{equation}
and note that with $L(t)$ given by equation \eqref{eq_Lcrit},
\begin{align}
f_0 t - \int\limits_0^ t \frac{\dot{L}(\zeta)^2 }{16D} d\zeta &=
\frac{c_{*}^2}{4D}t - \int\limits_0^ t \left( \frac{c_{*}^2}{4D} -\frac{\alpha c_{*}}{2D(\zeta+1)} +O\left(\frac{1}{(\zeta+1)^2}\right) + O(\dot{\eta}(\zeta)) \right) d\zeta \\
&=\frac{\alpha c_{*}}{2D}\log(t+1) + O(1) \qquad \textrm{as } t\rightarrow\infty
\end{align}

Consider $x=\frac{-L(t)}{2}+y$ with $y=O(1)$. Then $\xi=\frac{yL_0}{L(t)}=O\left(\frac{1}{t+1}\right)$ is certainly $O(P(t)^{-1/3} )$ as $t\rightarrow\infty$, and so
\begin{align}
\psi\left(\frac{-L(t)}{2}+y,t\right) &= w\left(\frac{yL_0}{L(t)}, t\right)\left(\frac{L_0}{L(t)}\right)^{1/2}\exp{\left(f_0 t - \int\limits_0^ t \frac{\dot{L}(\zeta)^2 }{16D} d\zeta - \frac{y}{L(t)}\left(\frac{y}{L(t)}-1\right) \frac{\dot{L}(t)L(t)}{4D}\right)} \nonumber \\
&=\overline{\underline{O}}\left(\frac{y}{t+1}\right) \times \frac{1}{(t+1)^{1/2}} \times \exp{\left(\frac{\alpha c_{*}}{2D}\log(t+1) - \frac{y^2}{4D(t+1)} +\frac{y c_{*}}{2D}\right)} \nonumber \\
&= \overline{\underline{O}}\left( y (t+1)^{-\frac{3}{2}+\frac{\alpha c_{*}}{2D}} \exp{\left(\frac{y c_{*}}{2D} \right)}\right)
\end{align}
which concludes the proof of Theorem \ref{theorem_crit}.
\end{proof}

\subsection{Critical case in higher dimensions} \label{section_crit_Rn}
To conclude this section we note that a similar analysis can also be applied to a ball in $\mathbb{R}^n$.
\begin{theorem}\label{theorem_RcritRn}
Let $\psi$ satisfy
\begin{equation}
\frac{\partial \psi}{\partial t} = D \nabla^2 \psi+ f_0\psi  \qquad \textrm{in } \{ \vert {\bf x}\vert < R(t)\} \subset \mathbb{R}^n
\end{equation}
\begin{equation}
\psi=0 \qquad \textrm{on } \vert {\bf x}\vert= R(t)
\end{equation}
where
\begin{equation}
R(t)=c_{*}t-\alpha \log(t+1)-\eta(t)
\end{equation}
with $\alpha>0$, and $\eta$ satisfying equation \eqref{eq_eta}, and where $n\leq 3$. Then
\begin{equation}
\psi({\bf x}, t) = \overline{\underline{O}} (y (t+1)^{-1-\frac{n}{2}+\frac{\alpha c_{*}}{2D}}) \qquad \textrm{as }t\rightarrow\infty \textrm{, for }y = R(t)-\vert {\bf x} \vert=O(1)
\end{equation}
\end{theorem}
\begin{remark}
Hence, the `critical value' of $\alpha$, for which the solution behaves exactly as order $y$, is now
\begin{equation}
\alpha_{crit}=\frac{(2+n)D}{c_{*}}
\end{equation}
As in the one-dimensional case, this appears to match the coefficient of the logarithmic correction term in the nonlinear FKPP problem on $\mathbb{R}^n$, with compactly supported initial conditions (see \cite{Gar}, \cite{RoqRosRou}).
\end{remark}
\begin{proof}
The change of variables ${\bf z}=\frac{{\bf x}}{R(t)}R_0$ and
\begin{equation}
W({\bf z}, t) = \psi({\bf x},t)\left(\frac{R(t)}{R_0}\right)^{\frac{n}{2}}\exp{\left(-f_0 t +\int\limits_0^t \frac{\vert \dot{R}(\zeta)\vert^2 }{4D}d\zeta+\frac{ \dot{R}(t)R(t)}{4DR_0^2}(r^2-R_0^2)\right)} 
\end{equation}
leads to the equation
\begin{equation}
\frac{\partial W}{\partial t} = D \frac{R_0^2 }{R(t)^2} \left( \nabla^2 W +  Q(t)\left(\frac{r^2}{R_0^2}-1\right)\frac{W}{R_0^2} \right) \qquad \textrm{on }  r < R_0
\end{equation}
\begin{equation}
W=0 \qquad\textrm{at } r=R_0
\end{equation}
where $r=\vert {\bf z} \vert$ and 
\begin{equation}
Q(t)=\frac{\ddot{R}(t)R(t)^3}{4D^2}
\end{equation}
which satisfies $Q(t)>0$, $Q(t)\rightarrow\infty$, $\dot{Q}\geq 0$ as $t\rightarrow\infty$.

Let $h({\bf x})=h_0(\vert {\bf x}\vert)$ be the radially symmetric principal eigenfunction of 
\begin{align}
\lambda h &= -\nabla^2 h \qquad \textrm{for }\vert {\bf x}\vert < 1, \\
h({\bf x}) &=0 \qquad \textrm{at }\vert {\bf x}\vert=1
\end{align}
in the $n$-dimensional ball, with eigenvalue $\lambda_0$. Then the function
\begin{equation}
\overline{W}(\vert {\bf z}\vert),t)=h_0\left(\frac{\vert {\bf z}\vert}{R_0}\right)\exp{\left(\int_0^{t} -\frac{D\lambda_0}{R(\zeta)^2}d\zeta \right)}
\end{equation}
is a supersolution for $W$. Thus, up to multiplication by a constant, $W({\bf z},t) \leq \overline{W}(\vert {\bf z}\vert,t)$.

Next consider the function
\begin{equation}
w_1(r,t)=\tilde{w}(R_0-r,t)
\end{equation}
where $\tilde{w}(\xi,t)= \underline{w}(\xi,t)a(t)$ is given in equations \eqref{eq_underlinew} and \eqref{eq_at} with $L(t)=2R(t)$, $L_0=2R_0$, $\xi=R_0-r$, and $P(t)=\frac{\ddot{L}(t)L(t)^3}{4D^2}$. Note that
\begin{align}
\left(\frac{r^2}{R_0^2}-1\right)\frac{Q(t)}{R_0^2} &=
\left(\frac{\xi-R_0}{R_0}+1\right)\left(\frac{\xi-R_0}{R_0}-1\right)\frac{1}{4R_0^2D^2}\ddot{R}(t)R(t)^3 \nonumber \\
&=\frac{2\xi}{L_0}\left(\frac{2\xi}{L_0}-2\right) \frac{1}{ L_0^2D^2}\frac{\ddot{L}(t)L(t)^3}{16} \nonumber \\
&= \frac{\xi}{L_0}\left(\frac{\xi}{L_0}-1\right) \frac{1}{L_0^2}\frac{\ddot{L}(t)L(t)^3}{4D^2} \nonumber \\
&= \frac{\xi}{L_0}\left(\frac{\xi}{L_0}-1\right) \frac{P(t)}{L_0^2}
\end{align}
Therefore (for $t$ large enough), this function 
$w_1(r,t)$ satisfies
\begin{equation}\label{w_1_ineq}
\frac{\partial w_1}{\partial t} \leq D \frac{R_0^2 }{R(t)^2} \left( \frac{\partial^2 w_1}{\partial r^2} + Q(t) \left(\frac{r^2}{R_0^2}-1\right)\frac{w_1}{R_0^2}\right) \qquad \textrm{on }  r < R_0
\end{equation}
Next, using the form of the Laplacian in $n$ dimensions, we claim that the function 
\begin{equation}
\hat{w}(r,t)=\frac{w_1(r,t)}{r^{\frac{n-1}{2}} }
\end{equation}
is a subsolution in the $n$-dimensional case when $n\leq3$. Certainly the boundary condition (at $r=R_0$) and the non-singular condition (at $r=0$) will be satisfied, since $w_1(R_0,t)=0$, and $w_1(r,t)=0$ on some neighbourhood $[0,r_0)$ of $r=0$. Moreover,
\begin{align}
\frac{\partial \hat{w}}{\partial t} = \frac{1}{r^{\frac{n-1}{2}} }\frac{\partial w_1}{\partial t} &\leq 
\frac{1}{r^{\frac{n-1}{2}} } D \frac{R_0^2 }{R(t)^2}\left( \frac{\partial^2 w_1}{\partial r^2} + Q(t) \left(\frac{r^2}{R_0^2}-1\right)\frac{w_1}{R_0^2}\right) \\
&= D \frac{R_0^2 }{R(t)^2}\left( \frac{1}{r^{\frac{n-1}{2}} }\frac{\partial^2 }{\partial r^2} \left(r^{\frac{n-1}{2}}\hat{w} \right) +Q(t)\left(\frac{r^2}{R_0^2}-1\right)\frac{\hat{w}}{R_0^2} \right)\\
&= D \frac{R_0^2 }{R(t)^2} \left( \nabla^2 \hat{w} +  \left(\frac{n-1}{2}\right)\left(\frac{n-3}{2}\right)\frac{\hat{w}}{r^2} +Q(t)\left(\frac{r^2}{R_0^2}-1\right)\frac{\hat{w}}{R_0^2}\right) \\
&\leq D \frac{R_0^2 }{R(t)^2} \left( \nabla^2 \hat{w} + Q(t)\left(\frac{r^2}{R_0^2}-1\right)\frac{\hat{w}}{R_0^2}\right)
\end{align}
where the equality follows from the form of the Laplacian in $n$ dimensions, and the final inequality holds for $n=1$,$2$,$3$. Thus, $\hat{w}$ is a subsolution and, up to multiplication by a constant, we obtain $\hat{w}\leq W$.

We are interested in $y = R(t)-\vert {\bf x} \vert=O(1)$. This corresponds to $\vert {\bf z} \vert = r=R_0-\frac{yR_0}{R(t)}$ with $y=O(1)$, for which the above supersolution and subsolution ($C_1\hat{w}\leq W \leq C_2\overline{W}$) provide the bounds $w({\bf z}, t)=\overline{\underline{O}}(R_0-\vert {\bf z} \vert ) = \overline{\underline{O}}\left(\frac{yR_0}{R(t)}\right)$, independently of $t$. Therefore, the same calculations as in the one-dimensional case give:
\begin{align}
\psi({\bf x} ,t)=&w({\bf z}, t)\left(\frac{R_0}{R(t)}\right)^{n/2}\exp{\left(f_0 t -\int\limits_0^t \frac{\vert \dot{R}(\zeta)\vert^2 }{4D}d\zeta-\frac{ \dot{R}(t)R(t)}{4DR_0^2}(r^2-R_0^2)\right)} \nonumber \\
&=\overline{\underline{O}} \left(\frac{y}{t+1}\right) \times \frac{1}{(t+1)^{n/2}} \times \exp{\left(\frac{\alpha c_{*}}{2D}\log(t+1) - \frac{y^2}{4D(t+1)} +\frac{y c_{*}}{2D} \right)} \nonumber \\
&= \overline{\underline{O}}\left( y (t+1)^{-1-\frac{n}{2}+\frac{\alpha c_{*}}{2D}} \exp{\left(\frac{y c_{*}}{2D} \right)} \right)
\end{align}
\end{proof}
\begin{remark}
The proof of Theorem \ref{theorem_RcritRn} shows that, in any dimension $n$,
\begin{equation}
\psi({\bf x}, t) = O (y (t+1)^{-1-\frac{n}{2}+\frac{\alpha c_{*}}{2D}}) \qquad \textrm{as }t\rightarrow\infty \textrm{, for }y = R(t)-\vert {\bf x} \vert=O(1)
\end{equation}
This follows from the supersolution. However the subsolution $\hat{w}$ used in the proof only satisfies the required inequality when $n\leq3$. One may conjecture that the full result of Theorem \ref{theorem_RcritRn} actually applies in all dimensions $n$.
\end{remark}

\section*{Acknowledgements}
I am very grateful to my PhD supervisor, Professor Elaine Crooks. I am also grateful for an EPSRC-funded studentship (project reference 2227486).

\end{document}